\documentclass[reqno]{structure} 

\usepackage{hyperref} 
\makeatletter \thm@headfont{\bfseries\scshape} \makeatother
\usepackage{latexsym,euscript}
\usepackage{amsbsy,amscd,amsfonts,amsmath,amsopn,amssymb,amstext,amsthm,amsxtra,}
\usepackage{epsf,epsfig,graphics,color,verbatim}
\usepackage{cite, float,multirow}
\usepackage{subfigure}
\theoremstyle{plain}
\newtheorem{theorem}{Theorem}[section]
\newtheorem{lemma}[theorem]{Lemma}
\newtheorem{corollary}[theorem]{Corollary}
\newtheorem{proposition}[theorem]{Proposition}

\theoremstyle{definition}
\newtheorem{definition}[theorem]{Definition}

\newtheorem*{example}{\sc{Example}} 

\newtheorem*{acknowledgement}{Acknowledgement}

\pagespan{1}{2}
\DOI{10.2478/UDT-2020-0000}

\begin{document}

\title
[ON SOME PROPERTIES OF IRRATIONAL SUBSPACES]
{On some properties of irrational subspaces}

\author{Vasiliy Neckrasov 
        }
\affil{Moscow State University, Moscow, Russia}
\address{Vasiliy Neckrasov\\
         Department of Number Theory\\
         Faculty of Mechanics and Mathematics \\
         Moscow State University \\
         Vorobiovy Gory  \\
         Moscow 119992 \\
         Russia}
\email{vneckrasov@gmail.com}
\def\shortauthors{V. Neckrasov}               

\keywords{Badly approximable matrices, completely irrational subspaces, ($\alpha, \beta$)-games}
\thanks {Supported by Basis Foundation under grant n. 20-8-2-12-1.}

\begin{abstract}
In this paper we discuss some properties of completely irrational subspaces. We prove that there exist completely irrational subspaces that are badly approximable and, moreover, sets of such subspaces are winning in different senses. We get some bounds for Diophantine exponents of vectors that lie in badly approximable subspaces that are completely irrational; in particular, for any vector $\xi$ from two-dimensional badly approximable completely irrational subspace of $\mathbb{R}^d$ one has $\hat{\omega}(\xi) \leq \frac{\sqrt{5} - 1}{2}$. Besides that, some statements about the dimension of subspaces generated by best approximations to completely irrational subspace easily follow from properties that we discuss.

\end{abstract}

\maketitle

\par 

\section{Introduction}

Our paper deals with some statements related to badly approximability of irrational subspaces in $\mathbb{R}^d$. Here in Introduction we give the definition of completely irrational linear subspace and discuss various aspects of Schmidt's games. In Section 2 we formulate and prove our statements about winning properties of the set of badly approximable completely irrational subspaces, and in Section 3 we give several applications.

\subsection{Irrational subspaces}

We consider linear subspaces in $\mathbb{R}^d$. Linear subspace $\mathcal{L}$ in $\mathbb{R}^d$ is called \textit{rational} if it has a basis (over $\mathbb{R}$) consisting only of vectors with integer coordinates. We define $n$-dimensional subspace $\mathcal{L}$ to be \textit{$m$-irrational} if its intersection with any $m$-dimensional rational subspace is just $ \{ 0 \}$, but $\mathcal{L}$ intersects some $(m+1)$-dimensional rational subspace. 

\begin{definition}

We call $n$-dimensional subspace $\mathcal{L}$ in $\mathbb{R}^d$, $n < d$, \textit{completely irrational} if it is $(d-n)$-irrational. 

\end{definition}

Here we should notice that $n$-dimensional $m$-irrational subspaces with $m>d-n$ do not exist.

A simple sufficient condition for $n$-dimensional subspace $\mathcal{L}$ to be completely irrational can be formulated in terms of Pl\"{u}cker coordinates. Namely, if Pl\"{u}cker coordinates of the subspace $\mathcal{L}$ $\subset$ $\mathbb{R}^d$ are lineary independent over $\mathbb{Q}$, then $\mathcal{L}$ is completely irrational. 

Indeed, let $\alpha^1 = (\alpha_1^1, ..., \alpha_d^1), ..., \alpha^n = (\alpha_1^n, ..., \alpha_d^n)$ be a basis of $n$-dimensional subspace $\mathcal{L}$ and integer vectors $B^1 = (b_1^1, ..., b_d^1), ..., B^n = (b_1^m, ..., b_d^m)$ form a basis of $m$-dimensional rational subspace $\mathcal{M}$, $m = d-n$. Subspaces $\mathcal{L}$ and $\mathcal{M}$ have a trivial intersection if and only if the matrix 

$$\Xi = 
\begin{pmatrix}
  \mathfrak{A} \\
  \mathfrak{B}
\end{pmatrix}
=
\begin{pmatrix}
  \alpha_1^1& ... & \alpha_d^1 \\
  ... & ... & ... \\
  \alpha_1^n& ... & \alpha_d^n \\
  b_1^1& ... & b_d^1 \\
  ... & ... & ... \\
  b_1^m& ... & b_d^m
\end{pmatrix}\!,
\text{where }
\mathfrak{A} = 
\begin{pmatrix}
  \alpha_1^1& ... & \alpha_d^1 \\
  ... & ... & ... \\
  \alpha_1^n& ... & \alpha_d^n 
\end{pmatrix}\!, 
\mathfrak{B} = 
\begin{pmatrix}
  b_1^1& ... & b_d^1 \\
  ... & ... & ... \\
  b_1^m& ... & b_d^m
\end{pmatrix}
$$
which rows are coordinates of basis vectors of subspaces $\mathcal{L}$ and $\mathcal{M}$ has maximal rank,  that is 
$$
{\rm rk}\, \Xi = n + m.
$$

In our case $n+m = d$, so it is equivalent to the condition ${ \rm det}\, \Xi \neq 0$. It is clear that ${ \rm det}\, \Xi$ is a linear combination of $n \times n$-minors of the submatrix $\mathfrak{A}$ with integer coefficients. These minors are just Pl\"{u}cker coordinates
$$
p_{i_1, ..., i_n}, \,\,\, 1 \leq i_1 < ... < i_n \leq d
$$
of the subspace $\mathcal{L}$; so, if they are linearly independent over $\mathbb{Q}$, for any rational subspace $M$ of dimension $m = d - n$  we have ${ \rm det}\, \Xi \neq 0.$

A similar sufficient condition may be formulated in a more general situation. Let $m+n \leq d$ and $I \subset \{1, 2, ..., d\}$ be a set consisting of $n+m$ elements. Consider a collection $S$ which consists of those Pl\"{u}cker coordinates $p_{i_1 ... i_n}$ of subspace $\mathcal{L}$ for which holds the condition $i_j \in I, \, \, j = 1, ..., n$. If $S$ consists of linearly independent over $\mathbb{Q}$ numbers, then subspace $\mathcal{L}$ is $m'$-irrational for some $m' \geq m$. 

\subsection{Schmidt's games and generalizations}

Here we briefly describe a classical game introduced by Schmidt. Basic notions and results can be found in \cite{Sch80}.

Let $0 < \alpha, \beta < 1 $. Suppose that two players Bob and Alice choose in turn a nested sequence of closed balls in $\mathbb{R}^d$:
$$
B_1 \supset A_1 \supset B_2 \supset ...
$$
with the property that the radii $\rho(A_i), \rho(B_i)$ of the balls $A_i$, $B_i$ satisfy
$$
\rho(A_i) = \alpha \rho(B_i), \, \, \rho(B_{i+1}) = \beta \rho(A_i) \, \, \, \text{for all } i = 1, 2, ... \, .
$$

A set $E \subset \mathbb{R}^d$ is called \textit{$(\alpha, \beta)$-winning} if Alice has a strategy guaranteeing that the intersection 
$$
  \bigcap\limits^{\infty}_{i=0} A_i
$$
belongs to $E$ no matter how Bob plays. A set $E \supset \mathbb{R}^d$ is called \textit{$\alpha$-winning} if it is $(\alpha, \beta)$-winning for all $0 < \beta < 1$.

The following useful results are due to Schmidt (links).

\begin{proposition}\label{cont}
Suppose that $2 \beta < 1 + \alpha \beta$. Then every $(\alpha, \beta)$-winning set has the power of the continuum. Moreover, for $\alpha > 0$ any $\alpha$-winning set has full Hausdorff dimension.
\end{proposition}

\begin{proposition}\label{int}
The intersection of countably many $\alpha$-winning sets is $\alpha$-winning.
\end{proposition}
The main tool that was used in Schmidt's proof is the following
\begin{proposition}{(Schmidt's escaping lemma.)}\label{esc}
Suppose that $0 < \alpha, \beta < 1$ and $\gamma = 1 + \alpha \beta - 2 \beta > 0$. Let $t$ be an integer with $(\alpha \beta)^t < \frac{1}{2} \gamma$ and $u$ a vector of length 1. Suppose a ball $A_k$ occurs in the $(\alpha, \beta)$-game. Then Bob can play so that (no matter how Alice plays) every point $x$ of $A_{k+t}$ satisfies 
$$
((x - o_k), u) > \frac{1}{2} \gamma \rho(A_k)
$$
where $o_k$ denotes the center of $A_k$.
\end{proposition}

There are various modifications of Schmidt's game. One of them, the hyperplane absolute winning game, was introduced in \cite{HAW}.

Fix $\beta < \frac{1}{3}.$ Firstly, Bob chooses a closed ball $B_1 = B_1(x_1, \rho_1)$ with center $x_1$ and radius $\rho_1$ in $\mathbb{R}^d$. Then in each stage of the game, after Bob chooses $B_i$, Alice chooses an affin subspace $L$ of dimension $d-1$ and removes its $\varepsilon$-neighbourhood $A_i$ from $B_i$, where $0 < \varepsilon < \beta \rho_i$ (and $\varepsilon$ can be different depending on $i$). Then Bob chooses the next ball $B_{i+1}$ with radius $\rho_{i+1} \geq \beta \rho_i$ and under condition
$$
B_{i+1} \subset B_i \setminus A_i.
$$
The set $S$ is said to be \textit{hyperplane $\beta$-absolute winning} if Alice has a strategy guaranteeing that 
$$
  \bigcap\limits^{\infty}_{i=0} B_i
$$
intersects $S$.
We say $S$ is \textit{hyperplane 
absolute winning} if it is \textit{hyperplane $\beta$-absolute winning} for every $0 < \beta < \frac{1}{3}$.

Analogous to Schmidt's game, HAW property holds under intersection of sets (the proof is given in \cite{HAW}).
\begin{proposition} \label{hawintersect}
The countable intersection of hyperplane absolute winning sets is hyperplane absolute winning.
\end{proposition}
An important fact connecting $\alpha$-winning sets in Schmidt's game and HAW sets is the following 
\begin{proposition}
HAW implies $\alpha$-winning for $\alpha < \frac{1}{2}$. 
\end{proposition}
It means that in some sense HAW-game is "more strong" then Schmidt's game, but for some $\alpha$ and $\beta$ it may happen that $(\alpha, \beta)$-winning property does not follow from HAW, so in our paper we consider both HAW game and Schmidt's game independently.

\subsection{Badly approximable subspaces}

Badly approximable systems of linear forms are defined as follows.
Let
$$
L_j(x) = \theta_j^1x_1 + ... + \theta_j^nx_n, \, \, \,j = 1, ..., m,
$$
where $x \in \mathbb{R}^n$ is a vector with coordinates $ x_1, ..., x_n$, be a system of linear forms with real coefficients and suppose there exists a constant $c = c(L_1, ..., L_m) > 0$ such that
$$
(\max(|x_1|, ..., |x_n|))^n (\max(\parallel L_1(x) \parallel, ..., \parallel L_m(x) \parallel))^m > c
$$
where $\parallel \xi \parallel$ denotes the distance from $\xi$  to the nearest integer. Then this system is called badly approximable.

Badly approximable system of linear forms corresponds to the matrix $\Theta = (\theta^i_j)$. As in the paper ~\cite{Sch1969}, we denote the set of matrices of badly approximable linear forms by $N(n, m)$.

\begin{proposition}{(Schmidt, 1969.)}\label{badwin}
The set $N(n, m)$ is $(\alpha, \beta)$-winning if 
\begin{equation}
2 \alpha < 1 + \alpha \beta. \label{ineq}
\end{equation}
In particular $N(n,m)$ is $\alpha$-winning for $\alpha \leq \frac{1}{2}$.
\end{proposition}
The following result was proved by Broderick, Fishman, Simmons in \cite{hawbad}.
\begin{proposition} \label{nhaw}
The set $N(n, m)$ is HAW.
\end{proposition}

It is well known that the set $N(n, m)$ has zero Lebesgue measure in $\mathbb{R}^{mn}$ (for example, this fact follows from Khintchine-Groshev theorem  \cite{gros}), while Schmidt was the first who proved by means of Proposition \ref{cont} that $N(n, m)$ has full Hausdorff dimension (~\cite{Sch1969}, Theorem 1).

We call $n$-dimensional linear subspace $\mathcal{L}$ in $\mathbb{R}^d$ \textit{badly approximable} if we can choose coordinates $x_1, ..., x_n; y_1, ..., y_m$ in $\mathbb{R}^d$, $m+n = d$, in such a way that $\mathcal{L}$ is defined by an equation 
\begin{equation}
y = \Theta x, \label{theta}
\end{equation}
where $\Theta$ is $m \times n$ matrix of badly approximable system of linear forms.

It is clear that $\mathcal{L}$ is badly approximable if and only if 

$$
\inf \limits_{x \in \mathbb{Z}^d \setminus \{0\}} (\max(|x_1|, ..., |x_n|))^n { \rm dist}( \mathcal{L}, x)^m > 0
$$
where ${ \rm dist}(\mathcal{L}, x)$ denotes (Euclidean) distance between subspace $\mathcal{L}$ and $x$.

\section{Badly approximability and irrationality}

\subsection{Formulations}

Talking about completely irrational subspaces, we didn't prove their existance explicitly. In fact even a stronger statement takes place.
\begin{theorem}\label{irrwin}
 Suppose that $\alpha$ and $\beta$ satisfy inequality (\ref{ineq}). Then the set $I(n,m)$ of all $n \times m$-matrices $\Theta$ defining $n$-dimensional completely irrational subspace of the form $(\ref{theta})$ in $\mathbb{R}^{n+m}$ is $(\alpha, \beta)$-winning. In particular $I(n,m)$ is $\alpha$-winning for $\alpha \leq \frac{1}{2}$.
\end{theorem}

It is clear that the set $I(n,m)$ has full Lebesgue measure. However, by means of Proposition \ref{badwin} and Proposition \ref{int}, from Theorem \ref{irrwin} we immediately obtain the following easy corollary which seems to be of importance.

\begin{theorem}
The set $N^{irr}(n,m)$ of $n \times (d-n)$-matrices defining badly approximable $n$-dimensional completely irrational subspaces is $\alpha$-winning if $\alpha \leq \frac{1}{2}$.
\end{theorem}
We should also immediately notice that there exist badly approximable subspace that are not completely irrational. The simplest example is as follows.

Let $n$ and $m$ be not coprime: $(n, m) = d > 1$. Suppose $n = n_1d, m = m_1d$ and $\Theta$ is a $n_1 \times m_1$ matrix of badly approximable system of linear forms. Then the matrix
$$\Xi = 
\begin{pmatrix}
  \Theta & 0 & 0 & ...& 0 \\
  0 & \Theta & 0 & ... & 0 \\
  ... & ... & ... & ... & .. \\
  0 & 0 & 0 & ... & \Theta
\end{pmatrix}, 
$$
where 0 denotes $n_1 \times m_1$ matrix with zeros as elements, obviously defines a badly approximable subspace that is obviously not completely irrational.

As $N^{irr}(n,m) \subset N(n,m)$, it has zero Lebesgue measure in $\mathbb{R}^{mn}$. Proposition \ref{cont} shows that $N^{irr}(n,m)$ has full Hausdorff dimension in $\mathbb{R}^{mn}.$

Analogous statements can be proved for HAW property.

\begin{theorem}\label{irrhaw}
For any $n, m \in\mathbb{N}$ the following statements hold.
\begin{enumerate}
    \item The set $I(n,m)$ is HAW, and so
    \item The set $N^{irr}(n,m)$ is HAW.
\end{enumerate}
\end{theorem}

\subsection{Manifold escaping lemma}
The following lemma is a natural generalization of Schmidt's escaping lemma (Proposition \ref{esc}) to the case of algebraic manifolds of any degree.
\begin{lemma}{(Algebraic manifold escaping lemma.)}\label{manesc}
Consider $(\alpha, \beta)$-game in $\mathbb{R}^r$, and suppose $0 < \alpha < 1, 0 < \beta < 1, \gamma = 1 + \alpha \beta - 2 \beta > 0$. Let $f \in \mathbb{R}[z_1, ..., z_r]$ be a nonzero polynomial and $\{f(z) = 0\}$ be the corresponding algebraic manifold $M$. Then Bob can play so that (no matter how Alice plays) for some $\epsilon > 0$ and for some $l$ any point $y \in W_l$ satisfies the inequality 
\begin{equation}
    {\rm dist}(y, M) > \varepsilon.\label{dist1}
\end{equation}
\end{lemma}

\begin{proof}
We'll prove this statement by induction on degree $s$ of polynomial $f(z).$ It is convinient instead of inequality (\ref{dist1}) to consider two inequalities 
\begin{equation}
 {\rm dist}(y, M) > \varepsilon \,\,\,\,\,\,\,\,\,\text{and}  \,\,\,\,\,\,\,\,\, |f(y)| > \varepsilon \label{dist2}
\end{equation}
simultaneously.

Base is given by Schmidt's escaping lemma (Proposition \ref{esc});

Step. Suppose our lemma is correct for all polynomials $f$ such that ${\rm deg} f \leq s-1.$ Consider the polynomials $\frac{\partial f}{\partial z_i}$ and corresponding manifolds $M_i$ and let $W_m$ be such a ball that 
$$
\forall z \in W_m, \,\, \forall i \,\, { \rm dist}(z, M_i)>\varepsilon_{s-1} \,\,\,\,\,\,\,\,\, \text{and} \,\,\,\,\,\,\,\,\, \Big| \frac{\partial f}{\partial z_i} \Big| > \varepsilon_{s-1},
$$
where $\varepsilon_{s-1}$ comes from the inductive assumption of the form (\ref{dist2}). As $f(z)$ is a polynomial, $\frac{\partial^k f}{\partial z_{i_1} ... \partial z_{i_k}}(z)$ are also polynomials and there are only finitely many nonzero of them. We can bound them in $W_m$ by one common constant $K$, so
$$\Big| \frac{\partial^k f}{\partial z_{i_1} ... \partial z_{i_k}}(z) \Big| < K \,\,\,\, \forall z \in W_m, \,\,\,\, \forall k, \,\,\,\, \forall i_1, ..., i_k.
$$

Take $\delta$ such that $K \sum\limits_{k=2}^{s+1} \frac{1}{k!}d^k \delta^k < \frac{1}{8} \alpha \beta \gamma \delta \varepsilon_{s-1}$ and choose $t$ with condition 

$$
\frac{1}{2} \alpha \beta \delta < \rho(W_t) \leq \frac{1}{2} \delta.
$$
If $W_t$ doesn't intersect $M$, it easily completes the proof. Otherwise choose a point $a$ belonging to both the ball $W_t$ and manifold $M$ and denote $\frac{\partial f}{\partial z_i}(a)$ by $b_i$. Then 

$$
f(z) = \sum\limits_{i=1}^{r} b_i (z_i - a_i) + \sum \frac{1}{k} \frac{\partial^k f}{\partial z_{i_1} ... \partial z_{i_k}} (a) (z_{i_1} - a_{i_1}) ... (z_{i_k} - a_{i_k}),
$$
and for all points $z = (z_1, ..., z_r)$ of the manifold $M$ in $W_t$ one has
$$
\Big|\sum\limits_{i=1}^{r} b_i (z_i - a_i)\Big| = \Big| \sum \frac{1}{k} \frac{\partial^k f}{\partial z_{i_1} ... \partial z_{i_k}} (a) (z_{i_1} - a_{i_1}) ... (z_{i_k} - a_{i_k}) \Big| < \frac{1}{8} \alpha \beta \gamma \delta \varepsilon_{s-1}.
$$

Denote the vector $\frac{b}{|b|}$ by $u$. We get
\begin{equation}
|(u, z-a)| < \frac{\alpha \beta \gamma \delta \varepsilon_{s-1}}{8|b|} < \frac{\alpha \beta \gamma \delta}{8}. \label{it}
\end{equation}
If $o$ is the center of $W_t$, (\ref{it}) can be written as
$$
-\frac{\alpha \beta \gamma \delta}{8} < (u, z-o) + (u, o-a) < \frac{\alpha \beta \gamma \delta}{8},
$$
and it follows (depending on the sign of $(u, o-a)$) that either $(u, x-o) < \frac{\alpha \beta \gamma \delta}{8}$, or $(-u, z-o) < \frac{\alpha \beta \gamma \delta}{8}$. 

Changing the sign of the vector $u$ if necessary, for $u$ and for any point $z \in M \bigcap W_t$ we have
\begin{equation}
    (u, z-o) < \frac{\alpha \beta \gamma \delta}{8}. \label{eq1st}
\end{equation}
By Schmidt's lemma Bob can play in such a way that for some $l$ we have 
\begin{equation} (u, y-o) > \frac{1}{2} \gamma \rho(W_t) > \frac{\alpha \beta \gamma \delta}{4} \,\,\,\,\,\,\,\,\,\,\,\, \forall y \in W_l. \label{eq2nd}
\end{equation}

From (\ref{eq1st}) and (\ref{eq2nd}) immediately follows the inequality $|y - z| > \frac{\alpha \beta \gamma \delta}{8} = \varepsilon_1.$
As $W_l$ is compact, $|f(z)|$ is bounded from below by some positive value $\varepsilon_2$ in $W_l$. Choosing the minimum of $\varepsilon_1$ and $\varepsilon_2$ as $\varepsilon$, we complete the proof.
\end{proof}

By means of Lemma \ref{manesc} we can easily prove Theorem \ref{irrwin}.

\begin{proof}{of Theorem \ref{irrwin}.}
 The matrix $\Theta$ which defines the subspace $\mathcal{L}$ has the form 

$$\Theta = 
\begin{pmatrix}
  \theta_1^1& \theta_1^2& ... & \theta_1^n \\
  \theta_2^1& \theta_2^2& ... & \theta_2^n \\
  ... & ... & ... & ... \\
  \theta_m^1& \theta_m^2& ... & \theta_m^n \\
\end{pmatrix}.
$$

As in Subsection 1.1, consider a new matrix 

$$\Xi = 
\begin{pmatrix}
  \theta_1^1& \theta_2^1& ... & \theta_m^1 & 1 & 0 & ... & 0 \\
  \theta_1^2& \theta_2^2& ... & \theta_m^2 & 0 & 1 & ... & 0\\
  ... & ... & ... & ... & ... & ... & ... & ... \\
  \theta_1^n& \theta_2^n& ... & \theta_m^n & 0 & 0 & ... & 1\\
  a_1^1 & a_2^1 & ... & a_m^1 & a_{m+1}^1 & a_{m+2}^1 & ... & a_d^1 \\
  ... & ... & ... & ... & ... & ... & ... & ...\\
  a_1^m& a_2^m& ... & a_m^m & a_{m+1}^m & a_{m+2}^m & ... & a_d^m
\end{pmatrix}. 
$$
The first $n$ rows of this matrix consist of the vectors which form a basis of our subspace $\mathcal{L}$. The last $m$ rows consist of the vectors which form a basis of a certain $m$-dimensional rational subspace  $\mathcal{M}$.
 Subspaces $\mathcal{L}$ and $\mathcal{M}$ intersect nontrivially if and only if ${ \rm det} \, \Xi = 0$. This means that $\Theta$ considered as $mn$-dimensional vector of $\mathbb{R}^{nm}$ belongs to a certain algebraic manifold $S_{\mathcal{M}}$. 
 
 There are countably many rational subspaces in $\mathbb{R}^r$, and each defines an algebraic manifold we need to escape from in our $(\alpha, \beta)$-game. Successively applying Lemma \ref{manesc} with $r = n \cdot m$ to all manifolds $S_{\mathcal{M}}$ for various $\mathcal{M}$, we get that the set
 $$
 \mathbb{R}^r \setminus \Big( \bigcap\limits_{\mathcal{M} \text{ is a rational subspace}} S_{\mathcal{M}} \Big) = I(n, m)
 $$
 is $(\alpha, \beta)$-winning for $\alpha$ and $\beta$ satisfying the inequality (\ref{ineq}).
\end{proof}

\subsection{On HAW-game}

It appears to be that a statement analogous to the Manifold escaping lemma (Lemma \ref{manesc}) is also true in case of HAW.

\begin{lemma}{(HAW algebraic manifold escaping lemma.)} Consider HAW-game. Let $f \in \mathbb{R}[z_1, ..., z_r]$ be a polynomial; $\{f(z) = 0\}$ --- algebraic manifold $M$. Then Alice can play so that (no matter how Bob plays) for some $\varepsilon > 0$ and for some $k$ any point $y \in W_k$ satisfies the following inequalities: ${\rm dist}(y, M) > \varepsilon$ and $|f(y)| > \varepsilon$ .
\end{lemma}
We need the following obvious statement.
\begin{lemma}\label{easy}
Let $R$ be a parameter, such that $\rho_1 \geq R.$ Alice can always play in such a way that for some $t$ the inequality $\beta R < \rho_t \leq R$ holds.
\end{lemma}

Proof of HAW algebraic manifold escaping lemma uses the same idea as for $(\alpha, \beta)$-game. We'll show only the parts of the proof that are different.
\begin{proof}
We'll prove this statement by induction on degree $s$ of polynomial $f(z).$

Base $(s=1)$: Alice can choose the affine subspace $f(z) = 0$ as $A_1$.

All further steps and notations until the choice of $\delta$ are the same as in Lemma \ref{manesc}. 

Here we take $\delta$ such that $K \sum\limits_{k=2}^{s+1} \frac{1}{k!}r^k \delta^k < \frac{1}{4} \beta^2 \delta \varepsilon_{s-1}$ and choose (according to Lemma \ref{easy}) $t$ under condition
$$
\frac{1}{2} \beta \delta < \rho_t \leq \frac{1}{2} \delta.
$$

Following the notation of previous proof, we come to the inequality

$$
|(u, z-a)| < \frac{\beta^2 \delta}{4}
$$
for all points of the manifold $M$ contained in the ball $B_t$.

This inequality means that all such points are contained in $\frac{\beta^2 \delta}{4}$-neighbourhood of the affine subspace $(u, z-a) = 0$. We choose it as an affine subspace $L$, and $A_i$ is its $\varepsilon_i \cdot \rho_t = \beta \cdot \rho_t > \frac{1}{2} \beta^2 \delta$ neighbourhood. The proof is completed. 
\end{proof}

The proof of statement Theorem \ref{irrhaw}, statement (1) repeats the proof of Theorem \ref{irrwin}. Statement (2) of Theorem \ref{irrhaw} follows from statement (1), Propositions \ref{nhaw} and \ref{hawintersect}.

\section{Applications}

\subsection{Some upper bounds for the exponents}

In this subsection we consider Diophantine exponents of vectors. 

We deal with nonzero vectros $\xi$ in $\mathbb{R}^d$, $d \geq 2$, so without loss of generality we may suppose that $\xi$ is of the form 
\begin{equation}
\xi = (\xi_1, \dots, \xi_{d-1}, 1).   \label{form}
\end{equation}
The vector $\xi$ of form (\ref{form}) is called \textit{totally irrational} if its coordinates are linearly independent over $\mathbb{Q}$. In other words, it means that $\xi$ does not belong to any proper rational subspace of $\mathbb{R}^d$.

The ordinary Diophantine exponent $\omega(\xi)$  of $\xi$ is the supremum of the set of all $\gamma > 0$ for which the inequality
$$
\max\limits_{j = 1, \dots, d-1} \parallel q \xi_j  \parallel \leq q^{-\gamma}
$$
has infinitely many integer solutions $q > 0$. 

The uniform Diophantine exponent $\hat{\omega}(\xi)$  of $\xi$ is the supremum of the set of all $\gamma > 0$ for which the system of inequalities 
$$
\max\limits_{j = 1, \dots, d-1} \parallel q \xi_j  \parallel \leq t^{-\gamma}, \,\,\,\,\,\, 0 < q \leq t,
$$
has an integer solution $q$ for all $t$ large enough. 

It is well-known that for irrational $\xi$ we have $\frac{1}{d-1} \leq \hat{\omega}(\xi) \leq 1$ and that $\omega(\xi) \geq \hat{\omega}(\xi)$.

Suppose now that $d \geq 3$ and $1 \leq n < d$. We consider a linear $n$-dimensional subspace $\mathcal{L}$ in $\mathbb{R}^d$. Let 
$$
w_{n,d} = \frac{n}{d-n}
$$
and let $W_{n,d}$ be the unique root of the equation
$$
x^d - w_{n,d}^{d-2} (1 + w_{n,d})x + w_{n,d}^{d-1} = 0
$$
in the interval $(0, w_{n,d})$.


The following Proposition \ref{klw} was proved in \cite{KlWeiss}. Here we should mention that the notation here differs from \cite{KlWeiss}. In particular, our constants $w_{d,n}$ and $W_{n,d}$ can be obtained from those in \cite{KlWeiss} by substitution

$$
s \mapsto n-1, \,\,\,\,\,\,\,\,\,\,\, n \mapsto d-1.
$$

\begin{proposition}\label{klw} Let $\mathcal{L}$ be an $n-dimensional$ badly approximable linear subspace of $\mathbb{R}^d$. Then
\begin{enumerate}
    \item for any $\xi \in \mathcal{L}$ of the form (\ref{form}) one has
    \begin{equation}\label{klw1}
        \hat{\omega}(\xi) \leq w_{n,d};
    \end{equation}
    \item for any totally irrational $\xi \in \mathcal{L}$ of the form (\ref{form}) one has
    \begin{equation}\label{klw2}
        \hat{\omega}(\xi) \leq W_{n,d}.
    \end{equation}
\end{enumerate}
\end{proposition}

As it was shown in \cite{KlWeiss}, inequality (\ref{klw1}) follows from
\begin{equation}\label{ord}
    \omega(\xi) \leq w_{n,d}.
\end{equation}
Inequality (\ref{klw2}) is a corollary of the following result from \cite{MosMar}.

\begin{proposition}\label{totirrat}
Let $G_d(\hat{\omega}(\xi))$ be the unique positive root of the equation
$$
x^{d-2} = \frac{\hat{\omega}(\xi)}{1 - \hat{\omega}(\xi)} (x^{d-3} + \dots + x + 1).
$$
Suppose $\xi$ is totally irrational. Then the inequality 
$$
\frac{\omega(\xi)}{\hat{\omega}(\xi)} \geq G_d(\hat{\omega}(\xi))
$$
holds.
\end{proposition}

Define ${\rm dim}_{\mathbb{Q}}(\xi)$ as the maximal number of linearly independent among the coordinates of vector $\xi$. The result of Proposition \ref{totirrat} can be generalized for not totally irrational vectors. The following statement and certain discussion can be found in \cite{Sch}.

\begin{proposition} \label{schl}
Suppose $\xi = (\xi_1, \dots, \xi_{d-1}, 1)$ and ${\rm dim}_{\mathbb{Q}}(\xi) = r \geq 3$. Then 
$$
\frac{\omega(\xi)}{\hat{\omega}(\xi)} \geq G_r(\hat{\omega}(\xi)).
$$
\end{proposition}

If badly approximable linear subspace $\mathcal{L}$ is completely irrational, we can give a stronger upper bound for the uniform exponent of all vectors $\xi \in \mathcal{L}$. 

Let $\mathfrak{W}_{n,d}$ be the unique root of the equation
$$
x^{d-n+1} - w_{n,d}^{d-n-1} (1 + w_{n,d})x + w_{n,d}^{d-n} = 0
$$
in the interval $(0, w_{n,d})$.

\begin{theorem}\label{ourbound}
Suppose $\mathcal{L}$ is an $n-dimensional$ badly approximable completely irrational linear subspace of $\mathbb{R}^d$. Then for any $\xi \in \mathcal{L}$ of the form (\ref{form}) one has
\begin{equation}\label{thm34}
    \hat{\omega}(\xi) \leq \mathfrak{W}_{n,d}.
\end{equation}
\end{theorem}
\noindent Particularly, if $\xi$ is totally irrational, it belongs to some $1$-dimensional completely irrational subspace, and we get (\ref{klw2}).
\begin{proof}
In view of the discussion above the proof is extremely simple. Consider any $\xi = (\xi_1, \dots, \xi_{d-1}, 1) \in \mathcal{L}$. It is clear that ${\rm dim}_{\mathbb{Q}}(\xi) \geq d-n$, so by Proposition \ref{schl} we have
$$
\frac{\omega}{\hat{\omega}} \geq G_{d-n}.
$$
It can be found in \cite{KlWeiss} that $G_{d-n}$ is a root of the polynomial 
$$
g(x) = (1-\hat{\omega}) x^{d-n} - x^{d-n-1} + \hat{\omega} = 0. 
$$
As in \cite{KlWeiss}, from this and (\ref{ord}) we come to an inequality
$$
(1-\hat{\omega}) w_{n,d}^{d-n} - w_{n,d}^{d-n-1} \hat{\omega} + \hat{\omega}^{d-n+1} \geq 0.
$$
This gives us (\ref{thm34}).
\end{proof}

To illustrate Theorem \ref{ourbound} we give the following simplest
\begin{example} \label{example}
Let $\mathcal{L}$ be a 2-dimensional (n=2) badly approximable linear subspace in $\mathbb{R}^4$. Then:

\begin{itemize}
    \item For any $\xi \in \mathcal{L}$ of the form (\ref{form}), from (\ref{klw1}) we have 
    $$
    \hat{\omega}(\xi) \leq w_{2, 4} = 1
    $$
    \item For any totally irrational $\xi \in \mathcal{L}$ of the form (\ref{form}) from (\ref{klw2}) we have 
    $$
    \hat{\omega}(\xi) \leq W_{2, 4} \approx 0,54 \dots
    $$
    \item If $\mathcal{L}$ is completely irrational, from Theorem \ref{ourbound} for any $\xi \in \mathcal{L}$ we have
    $$
    \hat{\omega}(\xi) \leq \mathfrak{W}_{2, 4} = \frac{\sqrt{5} - 1}{2} \approx 0,62 \dots
    $$
    where $\mathfrak{W}_{2, 4}$ is the only root of the equation 
    $$
    x^3 - 2x + 1
    $$
    in the interval $(0, w_{n,d}) = (0, 1).$
   It is stronger then the bound for an arbitrary badly approximable $\mathcal{L}$ (but, of course, weaker then for totally irrational vector $\xi$).
\end{itemize}
\end{example}
\subsection{About the dimension of subspace generated by best approximations}

Here we obtain some easy statements about dimension of subspace generated by best approximations to a completely irrational subspace. We follow the definitions and notation from \cite{lmos}. In particular, for a \textit{good} $n \times m$ matrix $\Theta$ the sequence of its \textit{best approximations} $\{ z_{\nu} \}_{\nu = 1}^{\infty}$ is unique and well defined. In fact, these vectors are the best approximations vectors for the linear subspace $\mathcal{L}$ of the form (\ref{theta}).

As in \cite{lmos}, for a good matrix $\Theta$ (or subspace $\mathcal{L}$) we define
$$
R(\Theta) = R(\mathcal{L}) = \min \{r: \text{ there exists a linear subspace } \mathcal{B} \subseteq \mathbb{R}^d, {\rm dim } \mathcal{B} = r, $$
$$
\text{ and } \nu_0 \in \mathbb{N}
\text{ such that } z_{\nu} \in \mathcal{B} \text{ for all } \nu \geq \nu_0 \}.
$$

If $\mathcal{L}$ is completely irrational, the following easy statement holds.

\begin{theorem} \label{dimofapp}
Suppose $\mathcal{L}$ is an $n$-dimensional good completely irrational subspace in $\mathbb{R}^d$ of form (\ref{theta}). Then 
$$
R(\Theta) \geq d-n+1.
$$
\end{theorem}

\begin{proof}
Suppose all the best approximation vectors to $\mathcal{L}$ (starting from some $\nu_0$) lie in $(d-n)$-dimensional (rational) subspace $\mathcal{B}$. Let $S$ be a unit sphere having $0$ as its center, $d = \min\limits_{x \in \mathcal{B} \cap S, y \in \mathcal{L} \cap S} |x - y| > 0.$ Then 
$$
{ \rm dist }(z_{\nu}, \mathcal{L}) \geq d \cdot |z_{\nu}| \xrightarrow[\nu \rightarrow \infty]{} \infty,
$$
 hence $\{ z_{\nu} \}$ is not a sequence of the best approximations. 
\end{proof}

As an example, consider 2-dimensional subspace $\mathcal{L}$ in $\mathbb{R}^4$. The following proposition contains in Corollary 4 of Theorem 7 from \cite{lmos}.

\begin{proposition}
Suppose $2 \times 2$-matrix $\Theta$ of two-dimensional subspace $\mathcal{L}$ in $\mathbb{R}^4$ is good and $\mathcal{L}$ is not contained in a rational 3-dimensional subspace of $\mathbb{R}^4$. Then $R(\Theta) = 2$ or $R(\Theta) = 4.$
\end{proposition}

As completely irrational two-dimensional subspace in $\mathbb{R}^4$ can not lie in any 3-dimensional rational subspace, from Theorem \ref{dimofapp} we immediately deduce

\begin{corollary}
Suppose $2 \times 2$-matrix $\Theta$ of two-dimensional completely irrational subspace $\mathcal{L}$ in $\mathbb{R}^4$ is good. Then $R(\Theta) = 4.$
\end{corollary}



\begin{acknowledgement}
The author thanks Nikolay Moshchevitin for much useful advice and support.

The paper was published with the financial support of the Ministry of Education and Science of the Russian Federation as part of the program of the Mathematical Center for Fundamental and Applied Mathematics under the agreement  075-15-2019-1621
\end{acknowledgement}


%
\end{document}